\documentclass[11pt,dvipsnames,svgnames,table]{article}
\usepackage[T1]{fontenc}

\usepackage{graphicx}
\usepackage{amsmath}
\usepackage{amsthm}
\usepackage{amsfonts}
\usepackage{amssymb}
\usepackage{thmtools}
\usepackage{mathtools}

\usepackage{tikz}
\usepackage{tikz-cd}

\usepackage{float}

\usepackage{setspace}
\usepackage{stfloats}
\usepackage{lmodern}
\usepackage{microtype}

\usepackage{xcolor}

\usepackage[shortlabels]{enumitem}
\setlist[itemize]{topsep=0ex,itemsep=0ex,parsep=0.4ex}
\usepackage{caption}
\setlist[enumerate]{topsep=0ex,itemsep=0ex,parsep=0.4ex}

\usepackage[unicode=true]{hyperref}
\hypersetup{
    colorlinks=true,
    breaklinks=true,
    linkcolor={blue},
    citecolor={blue},
    urlcolor={Navy},
    pdftitle={Dense Matchings in Graphs with Independence
            Number 2}, 
    pdfauthor={Jung Hon Yip} 
}

\usepackage[noabbrev,capitalise]{cleveref}
\crefname{lem}{Lemma}{Lemmas}
\crefname{thm}{Theorem}{Theorems}
\crefname{cor}{Corollary}{Corollaries}
\crefname{prop}{Proposition}{Propositions}
\crefname{conj}{Conjecture}{Conjectures}
\crefname{obs}{Observation}{Observations}
\crefname{quest}{Question}{Questions}
\crefname{open}{Open Problem}{Open Problems}
\crefname{claim}{Claim}{Claims}
\crefname{equation}{}{}

\usepackage{thm-restate}
\newcounter{thmcounter}
\renewcommand{\thefootnote}{\fnsymbol{footnote}}
\newtheorem{thm}[thmcounter]{Theorem}
\newtheorem{lem}[thmcounter]{Lemma}

\newtheorem{obs}[thmcounter]{Observation}

\newcounter{claimnumber}[thmcounter]

\newtheorem{claim}{Claim}[claimnumber]

\theoremstyle{definition}
\newtheorem{conj}[thmcounter]{Conjecture}

\usepackage[margin=32mm]{geometry}
\usepackage{float}

\DeclarePairedDelimiter{\floor}{\lfloor}{\rfloor}
\DeclarePairedDelimiter{\ceil}{\lceil}{\rceil}

\renewcommand{\geq}{\geqslant}
\renewcommand{\leq}{\leqslant}

\DeclareMathOperator{\had}{had}

\newcommand{\E}{\mathbb{E}}

\renewcommand{\Pr}{\mathbb{P}}
\DeclareMathOperator{\Var}{Var}
\DeclareMathOperator{\Cov}{Cov}

\DeclareMathOperator{\HC}{HC}
\DeclareMathOperator{\cm}{cm}

\renewcommand{\thefootnote}{\fnsymbol{footnote}}
\usepackage[bottom]{footmisc}

\usepackage{tabularx}
\usepackage{multirow}

\usepackage[longnamesfirst,numbers,sort&compress]{natbib}
\makeatletter
\def\NAT@spacechar{~}
\makeatother
\newcommand{\defn}[1]{\textcolor{green!50!black}{\textit{#1}}}
\newcommand{\mathdefn}[1]{\textcolor{green!50!black}{#1}}
\newcommand{\conjdefn}[1]{\textcolor{red!50!black}{#1}}

\begin{document}

\author{
    Jung Hon Yip\footnotemark[2]}

\footnotetext[2]{School of Mathematics, Monash University, Melbourne, Australia, \texttt{junghon.yip@monash.edu}. Supported by a Monash Graduate Scholarship.}


\title{\bf\boldmath Dense Matchings of Linear Size in Graphs \\ with Independence Number 2}

\maketitle

\begin{abstract}
    For a real number $c > 4$, we prove that every graph $G$ with $\alpha(G) \leq 2$ and $|V(G)| \geq ct$ has a matching $M$ with $|M| = t$ such that the number of non-adjacent pairs of edges in $M$ is at most:
    \begin{equation*}
        \left( \frac{1}{c\left(c-1\right)^2} + O_c\left(t^{-1/3} \right) \right) \binom{t}{2}.
    \end{equation*}
    This is related to an open problem of Seymour (2016) about Hadwiger's Conjecture, who asked if there is a constant $\varepsilon > 0$ such that every graph $G$ with $\alpha(G) \leq 2$ has $\had(G) \geq (\frac{1}{3} + \varepsilon) |V(G)|$.
\end{abstract}

\renewcommand{\thefootnote}
{\arabic{footnote}}


\section{\boldmath Introduction}
\label{s:intro}
A graph $H$ is a \defn{minor} of a graph $G$ if $H$ can be obtained from $G$ by deleting vertices and edges, and by contracting edges. The maximum integer $n$ such that the complete graph $K_n$ is a minor of $G$ is called the \defn{Hadwiger number} of $G$, denoted \defn{$\had(G)$}. A \defn{proper colouring} of a graph $G$ assigns a colour to each vertex so that adjacent vertices receive distinct colours. The minimum integer $k$ such that $G$ has a proper colouring with $k$ colours is called the \defn{chromatic number} of $G$, denoted \defn{$\chi(G)$}. The cardinality of the largest set of pairwise non-adjacent vertices is called the \defn{independence number} of $G$, denoted \defn{$\alpha(G)$}. For any standard graph theory notation not defined above, see \citet{Diestel05}\footnote{Unless stated otherwise, all graphs are simple and finite.}.

Hadwiger's Conjecture is widely considered to be one of the deepest unsolved problems in graph theory~\citep{Hadwiger43}. See \citep{Kawa15,Toft96,SeymourHC,CV20} for surveys. It asserts that $\had(G) \geq \chi(G)$ for every graph $G$. Since $\chi(G) \alpha(G) \geq |V(G)|$ for every graph $G$, Hadwiger's Conjecture would imply that $\had(G) \geq |V(G)|/\alpha(G)$. However, it is open if $\had(G) \geq |V(G)|/\alpha(G)$ for every graph $G$. Several results in this direction are known \citep{DM82,Fox10,BK11,bohme2011minors,KawaHadwigerOdd07,FradkinClique12}.  A classical result of \citet{DM82} states that:
\begin{thm}[\citet{DM82}]
    \label{thm:dm}
    For every graph $G$, $\had(G) \geq  |V(G)|/ (2\alpha(G) -1)$.
\end{thm}
When $\alpha(G) = 2$, \cref{thm:dm} implies $\had(G) \geq |V(G)|/3$, but in this case Hadwiger's Conjecture asserts that $\had(G) \geq |V(G)|/2$. This gap is large, and thus the restriction of Hadwiger's Conjecture to graphs with independence number $2$ is interesting.  For convenience, we denote this special case \nameref{conj:hc_chi}.

\begin{conj}[\conjdefn{$\HC_{\alpha = 2}$}]
    For every graph $G$ with $\alpha(G) \leq 2$, $\had(G) \geq \chi(G)$. \label{conj:hc_chi}
\end{conj}
\nameref{conj:hc_chi} is a key open case of Hadwiger's Conjecture~\citep{PST03,CS12,Blasiak07,CO08,Bosse19,NS26a}. Every graph $G$ with $\alpha(G) = 2$ has $\chi(G) \geq |V(G)|/2$, but the weakening of \nameref{conj:hc_chi} remains open:
\begin{conj}[\conjdefn{$\HC_{n/2}$}]
    For every graph $G$ with $\alpha(G) \leq 2$, $\had(G) \geq |V(G)|/2$. \label{conj:hc_half}
\end{conj}
If $G$ is a graph with $\alpha(G) \leq 2$, then $\chi(G) \geq |V(G)|/ 2$, and thus \nameref{conj:hc_chi} implies \nameref{conj:hc_half}. \citet*{PST03} proved that if $G$ is a minimal counterexample to \nameref{conj:hc_chi}, then $G$ is a counterexample to \nameref{conj:hc_half}. Therefore, \nameref{conj:hc_chi} and \nameref{conj:hc_half} are equivalent.

\citet{NS26a} proved a dense variant of \nameref{conj:hc_half}.
\begin{thm}[\citet{NS26a}]
    \label{thm:dense}
    Every graph $G$ with $\alpha(G) = 2$ contains a minor $H$ with $|V(H)| = \ceil{|V(G)|/ 2}$ and
    \begin{equation*}
        |E(H)| \geq (\gamma - o(1)) \binom{|V(H)|}{2},
    \end{equation*}
    where $\gamma \approx 0.986882$.
\end{thm}

Even the smallest improvement to \cref{thm:dm} in the $\alpha(G) = 2$ case would be interesting. In particular, \citet{SeymourHC} conjectured that the constant factor of $1/3$ in \cref{thm:dm} could be improved:
\begin{conj}[\conjdefn{$\HC_{\varepsilon}$}]
    \label{conj:hc_epsilon}
    There is a constant $\varepsilon > 0$ such that every graph $G$ with $\alpha(G) \leq 2$ satisfies $\had(G) \geq (\frac{1}{3} + \varepsilon) |V(G)|$.
\end{conj}
\nameref{conj:hc_epsilon} turns out to be equivalent to a conjecture about connected matchings, which we define below.

For disjoint subsets $A, B \subseteq V(G)$, $A$ and $B$ are \defn{adjacent} if some vertex in $A$ is adjacent to some vertex in $B$, otherwise $A$ and $B$ are \defn{non-adjacent}. For $M \subseteq E(G)$, let
$\mathdefn{V(M)} := \{ v \in V(G) : \text{$v$ is an endpoint of an edge $e \in M$} \}$. For an edge $e \in E(G)$, let $\mathdefn{V(e)}:= V(\{e\})$.
A matching $M \subseteq E(G)$ is \defn{connected} if for every two edges $e, f \in M$, $V(e)$ and $V(f)$ are adjacent. By contracting the edges of a connected matching $M$ of $G$ and deleting the unmatched vertices, one obtains a $K_{|M|}$-minor of $G$. For a graph $G$, let $\mathdefn{\cm(G)}$ be the largest size of a connected matching in $G$. \cref{conj:linear_cm} is due to \citet{furedi2005connected} (although they note that other authors may have independently made the same conjecture).
\begin{conj}[\conjdefn{Linear-CM}]
    \label{conj:linear_cm}
    There is a constant $c > 0$ such that for every integer $t \geq 1$, every graph $G$ with $\alpha(G) \leq 2$ and $|V(G)| \geq ct$ has a connected matching of size $t$.
\end{conj}
Thomass\'e first noted that \nameref{conj:linear_cm} and \nameref{conj:hc_epsilon} are equivalent (see \citep{furedi2005connected}). A proof is given by \citet{KPT05}. Therefore, improving the constant $1/3$ in \cref{thm:dm} is as hard as finding a connected matching of linear size. See \citep{furedi2005connected,cambie2021hadwiger,chen2025connected,Fox10,Blasiak07} for partial results on \nameref{conj:linear_cm}. We remark that there is a precise version of \nameref{conj:linear_cm} conjectured by \citet{furedi2005connected}, who asked if for every integer $t \geq 1$, every graph $G$ with $\alpha(G) \leq 2$ and $|V(G)| \geq 4t - 1$ has a connected matching of size at least $t$. The bound $4t - 1$ is best possible, since the disjoint union of two copies of the complete graph $K_{2t - 1}$ does not have a connected matching of size $t$.

The goal of this paper is to prove a dense variant of \nameref{conj:linear_cm}, \cref{thm:main}. The proof of \cref{thm:main} is based on ideas from the proof of \cref{thm:dense} by \citet{NS26a}.
\begin{restatable}{thm}{dense}
    \label{thm:main}
    Let $c > 4$ be a real number, and $t \geq 1$ be an integer. Let $G$ be a graph with $\alpha(G) \leq 2$ and $|V(G)| \geq ct$. Then there is a matching $M$ of $G$ with $|M| = t$ such that the number of non-adjacent pairs of edges in $M$ is at most
    \begin{equation*}
        \left( \frac{1}{c\left(c-1\right)^2} + O_c\left(t^{-1/3} \right) \right) \binom{t}{2}.
    \end{equation*}
\end{restatable}
\cref{thm:main} implies that for any $\varepsilon > 0$, there exists a $c := c(\varepsilon)$ such that for every integer $t \geq 1$, every graph $G$ with $\alpha(G) \leq 2$ and $|V(G)| \geq ct$ contains a matching $M$ with $|M| = t$ such that the density of adjacent edge pairs is at least $1 - \varepsilon$.
Similar results on dense minors have been studied \citep{NS26a,NguyenDense22,HNSTDense25}. In particular, \citet{NguyenDense22} proved that for every $\varepsilon > 0$, there exists $k > 0$ such that for every integer $t \geq 2$, every graph with chromatic number $kt$ contains a minor with $t$ vertices and edge density at least $1 - \varepsilon$. However, this does not imply \cref{thm:main}, since we require each branch set to have two vertices. \citet{Fox10} used probabilistic methods to prove that every sufficiently large\footnote{\nameref{conj:linear_cm} is equivalent to the assertion that $\cm(G) \geq \Omega(|V(G)|)$.} graph $G$ with $\alpha(G) = 2$ has a connected matching of size at least $\Omega(n^{4/5} \log^{1/5}n)$, where $n = |V(G)|$. However, this does not imply \cref{thm:main}.
\section{Proof}

We use $\mathdefn{\omega(G)}$ to denote the \defn{clique number}, $\mathdefn{\delta(G)}$ to denote the \defn{minimum degree}, and $\mathdefn{\Delta(G)}$ to denote the \defn{maximum degree} of $G$. Let $\mathdefn{\overline{G}}$ denote the \defn{complement} of $G$.
We begin with a lemma, proved by \citet{cambie2021hadwiger}. For completeness, we include the proof here.

\begin{lem}[\citet{cambie2021hadwiger}]
    \label{lem:cambie}
    For every integer $t \geq 1$, if $G$ is a graph with $\alpha(G) = 2$, $|V(G)| \geq 4t - 1$, and $\cm(G) \leq t - 1$, then $\omega(G) \leq \cm(G)$.
\end{lem}
\begin{proof}
    If $G$ is disconnected, then $G$ is the union of two complete graphs. One component of $G$ has order at least $2t$, implying $\cm(G) \geq t$.
    Hence $G$ is connected.
    Note that $\cm(G) \geq \floor{\omega(G) /2}$, so $\omega(G) \leq 2t-1$. Since $\alpha(G) \leq 2$, the non-neighbours of each vertex of $G$ form a clique, so $\delta(G) \geq |V(G)|-1 - \omega(G)\geq  2t-1$.
    Let $A$ be a clique of order $\omega(G)$ and let $B=V(G) \setminus A.$
    Let $M$ be the largest matching in the bipartite subgraph induced by the edges between $A$ and $B$. Any matching from $A$ to $B$ is connected since $A$ is a clique, so $\left|M\right| \leq \cm(G) \leq t-1$.
    Suppose for contradiction that $\omega(G) > \cm(G) \geq |M|$.
    This implies there is a vertex $x \in A$ unmatched by $M$. However, there is no neighbour of $x$ contained in $B - V(M)$, otherwise this contradicts the maximality of $M$. Therefore, all neighbours of $x$ are contained in $A \cup V(M)$. Since $\deg_G(x) \geq 2t - 1$, $|A \cup V(M)| \geq 2t$. Observe that
    \begin{equation*}
        2|M| + |A \setminus V(M)| = |V(M) \cup (A \setminus V(M))| = |A \cup V(M)| = |A| + |M| \geq 2t.
    \end{equation*}
    Therefore, $|A \setminus V(M)| \geq 2t - 2|M|$, implying there are at least $2t - 2|M|$ vertices in $A$ left unmatched by $M$. We can extend the matching $M$ by pairing the remaining unmatched vertices of $A$, obtaining a connected matching of size $t$, a contradiction.
\end{proof}

Let $S$ be a finite set with $|S|$ even. Let \defn{$\binom{S}{2}$} denote the set of all pairs\footnote{In this paper, a \defn{pair} is an unordered two-element set of distinct elements.} of elements from $S$. A \defn{random partition of $S$ into pairs} is a partition of $S$ into pairs, chosen uniformly at random from the set of all partitions of $S$ into pairs. Let $X$ be a random partition of $S$ into pairs. Note that $|\binom{S}{2}| = \binom{|S|}{2}$, $|X| = |S|/2$, and that for each $e \in \binom{S}{2}$ and every $f \in \binom{S}{2}$ disjoint from $e$,
\begin{equation}
    \label{eq:MBounds}
    \Pr(e \in X) =\frac{1}{|S|-1} \qquad \text{and} \qquad \Pr(e,f \in X) = \frac{1}{(|S|-1)(|S|-3)}.
\end{equation}

For $F \subseteq \binom{S}{2}$, \cref{lem:chebyshev} bounds the concentration of $|F \cap X|$ around $|F|/(|S| - 1)$. \cref{lem:chebyshev} was first proved by \citet{NS26a}, but for completeness we provide the proof here.
\begin{lem}[\citet{NS26a}]
    \label{lem:chebyshev}
    Let $S$ be a finite set with $|S| \geq 4$ even, and let $X$ be a random partition of $S$ into pairs. Then for every $F \subseteq \binom{S}{2}$ and every real number $\lambda>0$,
    \begin{equation*}
        \Pr\left(\left||F \cap X| - \frac{|F|}{|S|-1}\right| \geq \lambda \right) \leq \frac{|S|}{\lambda^2}.
    \end{equation*}
\end{lem}
\begin{proof}
    For $\lambda > 0$, let $E$ be the event that
    \begin{equation*}
        \frac{|F|}{|S| - 1} - \lambda < |F \cap X| < \frac{|F|}{|S| - 1} + \lambda.
    \end{equation*}
    Let $F^C := \binom{S}{2} \setminus F$, so that $|F^C \cap X|$ counts the number of elements of $X$ not in $F$. Then $E$ occurs if and only if (noting that $|X| = |S|/2$):
    \begin{equation*}
        \frac{|S|}{2} -  \frac{|F|}{|S| - 1} - \lambda \leq |F^C \cap X| \leq \frac{|S|}{2} -  \frac{|F|}{|S| - 1} + \lambda.
    \end{equation*}
    Expanding,
    \begin{equation*}
        \frac{|S|(|S| - 1) - 2|F|}{2(|S| - 1)} - \lambda  \leq  |F^C \cap X| \leq \frac{|S|(|S| - 1) - 2|F|}{2(|S| - 1)} + \lambda.
    \end{equation*}
    However, $|S|(|S| - 1) - 2|F| = 2|\binom{S}{2} \setminus F| = 2 |F^C|$. Thus $E$ occurs if and only if
    \begin{equation*}
        \frac{|F^C|}{|S| - 1} - \lambda \leq |F^C \cap X| \leq \frac{|F^C|}{|S| - 1} + \lambda.
    \end{equation*}
    Therefore, suppose event $E$ occurs. By replacing $F$ with $F^C$ if needed, we may without loss of generality assume that
    \begin{equation*}
        |F| \leq \frac{|S|(|S| - 1)}{4}.
    \end{equation*}
    For each pair $e \in F$, let $Z_e$ be the indicator random variable with $Z_e = 1$ if $e \in X$ and $Z_e = 0$ otherwise. Let $Z = \sum_{e \in F} Z_e = |F \cap X|$.
    By \cref{eq:MBounds}, $\E(Z) = |F|/(|S| - 1)$ and
    \begin{align*}
        \Var(Z) & = \sum_{(e,f) \in F \times F} \Cov(Z_e, Z_f)                                                                                                                                                                \\
        \intertext{Note that if $e \cap f \neq \varnothing$, then $e$ and $f$ cannot be both in $X$. Hence, $\E(Z_e \cdot Z_f) = 0$. By definition of covariance, $\Cov(Z_e, Z_f) = \E(Z_e \cdot Z_f) - \E(Z_e) \cdot \E(Z_f) \leq 0$. Therefore,}
        \Var(Z) & \leq \sum_{e \in F} \Cov(Z_e, Z_e) + \sum_{e \in F} \left(\sum_{f \in F,\, f \cap \, e = \varnothing} \Cov(Z_e, Z_f)  \right)                                                                               \\
                & = |F| \cdot \Var(Z_e) + \sum_{e \in F} \left( \sum_{f \in F,\, f \cap \, e = \varnothing} \E(Z_e \cdot Z_f) - \E(Z_e) \cdot \E(Z_f) \right)                                                                 \\
                & \leq |F| \cdot \Pr(Z_e = 1) + \sum_{e \in F} \left(\sum_{f \in F,\, f \cap \, e = \varnothing} \left( \frac{1}{\left(|S| - 1\right)\left(|S| - 3\right)} - \left( \frac{1}{|S| - 1}\right)^2\right) \right) \\
                & \leq \frac{|F|}{|S| - 1} + |F|^2 \left( \frac{1}{\left(|S| - 1\right)\left(|S| - 3\right)} - \left( \frac{1}{|S| - 1}\right)^2\right)                                                                       \\
                & \leq \frac{|S|}{4} + \frac{|S|^2}{8(|S| - 3)} \leq |S|.
    \end{align*}
    By Chebyshev's inequality (see for instance \citep{AS00}, Theorem 4.1.1),
    \begin{equation*}
        \Pr\left(\left||F \cap X| - \frac{|F|}{|S| - 1} \right| \geq \lambda \right) \leq \left(\frac{\sqrt{\Var(Z)}}{\lambda}\right)^2 \leq \frac{|S|}{\lambda^2},
    \end{equation*}
    as desired.
\end{proof}

Let $G$ be a graph. An ordered sequence of four distinct vertices $(u,v,w,z)$ is a \defn{bad quadruple} if $uv, wz \in E(G)$, and $uw, uz, vw, vz \notin E(G)$. Observe that every pair of non-adjacent edges that do not share an endpoint corresponds to eight bad quadruples.
\begin{lem}
    \label{lem:bad_quads}
    Let $k \geq 1$ be an integer, and let $G$ be a graph such that $|E(\overline{G})| = b$, and suppose $\delta(G) \geq |V(G)| - k$. Then the number of bad quadruples is at most $2b(k-1)^2$.
\end{lem}
\begin{proof}
    Let $(u,v,w,z)$ be a bad quadruple.
    There are $2|E(\overline{G})| = 2b$ ways to choose $(u,w)$ with $uw \notin E(G)$. For a fixed $(u,w)$, $v \notin \{ u, w \}$ is chosen from at most $k-1$ non-neighbours of $w$, and $z \notin \{ u, v, w \}$ is chosen from at most $k-1$ non-neighbours of $u$. Thus, the number of bad quadruples is at most $2b(k-1)^2$.
\end{proof}

\begin{lem}
    \label{lem:tech_bound}
    For an integer $t \geq 1$,
    let $G$ be a graph with $\alpha(G) \leq 2$ and $|V(G)|$ even. Suppose $c \geq 4$ is a real number such that $|V(G)| = ct$ and suppose $\cm(G) \leq t - 1$. Let $\ell$ be a real number such that $\ell^2 > c / t$ and $\ell \leq \frac{1}{2}c - \frac{3}{2}$. Let $q:= 1 - \frac{c}{\ell^2 t}$, let $k:= \frac{1}{2}(c-1) - \ell$, so that $q > 0$ and $k \geq 1$, and let $p:=1/k$. Then there is a matching $M$ of $G$ with $|M| = t$ such that the number of non-adjacent pairs of edges is at most
    \begin{equation*}
        \frac{p^2ct(t-1)^3}{8q(ct-1)(ct-3)}.
    \end{equation*}
\end{lem}

\begin{proof}
    Let \defn{$\mathcal{A}_{\ell}$} be the set of all partitions $X$ of $V(G)$ into pairs such that
    \begin{equation*}
        |X \cap E(G)| \geq \frac{1}{2}(c-1)t - \ell t.
    \end{equation*}
    Since $|V(G)| \geq ct$ and $\cm(G) \leq t - 1$, by \cref{lem:cambie}, $\omega(G) \leq t - 1$. Since $\alpha(G) \leq 2$, the non-neighbours of each vertex form a clique, which implies that $\Delta(\overline{G}) \leq t - 1$ and $\delta(G) \geq |V(G)| - 1 - (t - 1) \geq (c-1)t$.
    \begin{claim}
        \label{claim:al_non_empty}
        $\mathcal{A}_{\ell} \neq \varnothing$.
    \end{claim}
    \begin{proof}
        Let $X$ be a random partition of $V(G)$ into pairs.
        By taking $F = E(G)$ and $S = V(G)$ in \cref{lem:chebyshev},
        \begin{equation*}
            \Pr\left( \left| \, | E(G) \cap X | - \frac{|E(G)|}{|V(G)| - 1} \, \right| \geq \ell t\right) \leq \frac{|V(G)|}{\ell^2 t^2}.
        \end{equation*}
        Therefore,
        \begin{equation}
            \label{eq:bound_first_choice}
            \Pr\left(\frac{|E(G)|}{|V(G)| - 1} - \ell t  \leq |E(G) \cap X| \leq \frac{|E(G)|}{|V(G)| - 1} + \ell t\right) \geq 1 - \frac{|V(G)|}{\ell^2 t^2} = q > 0.
        \end{equation}
        Note that when $|X \cap E(G)| \geq \frac{|E(G)|}{|V(G)| - 1} - \ell t$,
        \begin{equation*}
            |X \cap E(G)| \geq \frac{|E(G)|}{|V(G)| - 1} - \ell t \geq \frac{|V(G)|(c-1)t}{2(|V(G)| - 1)} - \ell t \geq  \frac{1}{2}(c-1)t - \ell t,
        \end{equation*}
        and therefore $X \in \mathcal{A}_{\ell}$. By \cref{eq:bound_first_choice},
        \begin{equation}
            \label{eq:a_l}
            \Pr(X \in \mathcal{A}_{\ell}) \geq q > 0.
        \end{equation}
        Thus, $\mathcal{A}_{\ell} \neq \varnothing$.
    \end{proof}
    We now describe our process of choosing a random matching $M$ of size $t$ in $G$.
    Choose $X \in \mathcal{A}_{\ell}$ (which by \cref{claim:al_non_empty} is non-empty) uniformly at random. Then, choose a matching $M$ with $t$ edges from $X \cap E(G)$; such a choice is possible because $|X \cap E(G)| \geq \frac{1}{2}(c-1)t - \ell t  = kt \geq t$.
    Observe that each edge of $X \cap E(G)$ is chosen to be in $M$ with probability at most $p = 1/k$, and therefore for any edge $e \in E(G)$:\footnote{Note that we have described two random processes. The first defines $X$ to be a random partition of $V(G)$ into pairs. The second chooses $X \in \mathcal{A}_{\ell}$ uniformly at random. The probabilities in \cref{claim:conditional} refer to the second process. From now on, we use $\Pr(\,\cdots \mid X \in \mathcal{A}_{\ell}$) to emphasise that these probabilities refer to the second process. For probabilities in the first process, we do not condition on $X \in \mathcal{A}_{\ell}$.}
    \begin{equation*}
        \Pr(e \in M \mid e \in X \text{ and } X \in \mathcal{A}_{\ell}) = \frac{t}{|X \cap E(G)|} \leq \frac{t}{kt} = \frac{1}{k} = p.
    \end{equation*}
    For any two distinct edges $e, f \in E(G)$ that do not share an endpoint:
    \begin{align}
        \label{eq:ef_in_M}
        \Pr(e,f \in M \mid e \in X \text{ and } X \in \mathcal{A}_{\ell}) & = \frac{t}{|X \cap E(G)|} \cdot \frac{t-1}{|X \cap E(G)|-1} \nonumber               \\
                                                                          & \leq \frac{t}{kt} \cdot \frac{t-1}{kt - 1} \leq \left( \frac{t}{kt}\right)^2 = p^2.
    \end{align}
    \begin{claim}
        \label{claim:conditional}
        For any two distinct edges $e, f \in E(G)$ that do not share an endpoint,
        \begin{equation*}
            \Pr(e \in X \mid X \in \mathcal{A}_{\ell}) \leq \frac{1}{q(ct - 1)}, \quad \text{and} \quad \Pr(e, f \in X \mid X \in \mathcal{A}_{\ell}) \leq \frac{1}{q(ct - 1)(ct- 3)}.
        \end{equation*}
    \end{claim}
    \begin{proof}
        Consider a random partition $X$ of $V(G)$ into pairs. By \cref{eq:a_l}:
        \begin{align*}
            \Pr(e \in X) & \geq \Pr(e \in X \text{ and } X \in \mathcal{A}_{\ell})                          \\
                         & = \Pr(e \in X \mid X \in \mathcal{A}_{\ell}) \cdot \Pr(X \in \mathcal{A}_{\ell}) \\
                         & \geq \Pr(e \in X \mid X \in \mathcal{A}_{\ell}) \cdot q.
        \end{align*}
        By \cref{eq:MBounds},
        \begin{equation*}
            \Pr(e \in X) = \frac{1}{|V(G)| - 1} \leq \frac{1}{ct - 1}.
        \end{equation*}
        Rearranging gives the desired result. The second claim follows from a similar argument.
    \end{proof}
    Since $\Delta(\overline{G}) \leq t - 1$, $|E(\overline{G})| \leq \frac{1}{2}ct(t-1)$. By \cref{lem:bad_quads}, the number of bad quadruples of $G$ is at most $ct(t-1)^3$. Therefore, there are at most $ct(t-1)^3 /8$ non-adjacent pairs of edges in $G$ that do not share an endpoint. Define:
    \begin{equation*}
        \mathdefn{\mathcal{U}}:= \left\{\{e,f\} \in \binom{E(G)}{2}: \text{$e$ and $f$ do not share an endpoint and are non-adjacent}\right\}.
    \end{equation*}
    Then
    \begin{equation}
        \label{eq:U_bound}
        |\mathcal{U}| \leq \frac{ct(t-1)^3}{8}.
    \end{equation}

    For any two distinct edges $e, f \in E(G)$ that do not share an endpoint, if $e, f \in M$, then $e, f \in X$, and thus:
    \begin{align*}
        \Pr(e, f \in M \mid X \in \mathcal{A}_{\ell}) & = \Pr(e, f \in M \text{ and } e, f \in X \mid X \in \mathcal{A}_{\ell}).
    \end{align*}
    By \cref{eq:ef_in_M,claim:conditional},
    \begin{align}
        \label{eq:ef_in_M2}
        \Pr(e, f \in M \mid X \in \mathcal{A}_{\ell}) & = \frac{\Pr(e, f \in M \text{ and } e, f \in X \mid X \in \mathcal{A}_{\ell})}{\Pr(e, f \in X \mid X \in \mathcal{A}_{\ell})} \cdot \Pr(e, f \in X \mid X \in \mathcal{A}_{\ell}) \nonumber \\
                                                      & = \Pr(e, f \in M \mid e, f \in X \text{ and } X \in \mathcal{A}_{\ell}) \cdot \Pr(e, f \in X \mid X \in \mathcal{A}_{\ell}) \nonumber                                                       \\
                                                      & \leq p^2 \cdot \frac{1}{q(ct - 1)(ct- 3)}.
    \end{align}

    From \cref{eq:ef_in_M2,eq:U_bound}, it follows that
    \begin{align*}
        \E(\text{number of non-adjacent pairs of edges in $M$} \mid X \in \mathcal{A}_{\ell}) & = \sum_{\{e, f\} \, \in \, \mathcal{U}} \Pr(e, f \in M \mid X \in \mathcal{A}_{\ell}) \\
                                                                                              & \leq \frac{p^2ct(t-1)^3}{8q(ct-1)(ct-3)}.
    \end{align*}
    This implies that there is a matching $M$ of $G$ with $|M| = t$ such that the number of non-adjacent pairs of edges is at most $\frac{p^2ct(t-1)^3}{8q(ct-1)(ct-3)}$, as desired.
\end{proof}
We minimise the bound in \cref{lem:tech_bound} for $t$ large relative to $c$. We use the following observation repeatedly.
\begin{obs}
    Let $f$ and $g$ be monic polynomials with $\deg(f) = n$ and $\deg(g) = m$. Then
    \begin{equation*}
        \frac{f(x)}{g(x)} = x^{n-m} + O(x^{n-m-1}).
    \end{equation*}
\end{obs}

We now finish the proof of the main theorem, \cref{thm:main}.

\dense*
\begin{proof}
    If $\cm(G) \geq t$, then we are done. Otherwise, if $|V(G)|$ is odd, then let $G'$ be the graph obtained from $G$ by deleting an arbitrary vertex from $G$. If $|V(G)|$ is even, then let $G':= G$. Therefore, $|V(G')|$ is even.
    Let $c' > 0$ be a constant such that $c't = |V(G')|$. Then $c't \geq |V(G')| \geq |V(G)| - 1 \geq ct - 1$, so $c' \geq c - \frac{1}{t} $.  When $t$ is large relative to $c$, since $c > 4$, we have $c' \geq 4$. Applying \cref{lem:tech_bound} to $G'$, for every real number $\ell$ such that $\ell^2 > c'/t$ and $\ell \leq \frac{1}{2}c' - \frac{3}{2}$, there is a matching $M$ of $G'$ with $|M| = t$ such that at most
    \begin{equation*}
        \frac{p^2c't(t-1)^3}{8q(c't-1)(c't-3)}
    \end{equation*}
    pairs of edges are non-adjacent, where $k:= \frac{1}{2}(c'-1) - \ell$, $q:= 1 - \frac{c'}{\ell^2 t}$ and $p:=\frac{1}{k} $. Plugging these in yields
    \begin{align*}
        \frac{p^2c't\left(t-1\right)^3}{8q\left(c't-1\right)\left(c't-3\right)} & = \frac{\left(\frac{1}{\frac{1}{2}\left(c'-1\right) - \ell} \right)^2c't\left(t-1\right)^3}{8\left(1 - \frac{c'}{\ell^2 t}\right)\left(c't-1\right)\left(c't-3\right)} \\
                                                                                & = \frac{c't\left(t-1\right)^3}{2\left(c'-1-2\ell\right)^2 \left(1 - \frac{c'}{\ell^2 t}\right)\left(c't - 1\right)\left(c't - 3\right) }:= \mathdefn{f(c', t, \ell)}.
    \end{align*}
    This expression is minimised when
    \begin{equation}
        \label{eq:ell_value}
        \ell := \left(\frac{c'\left(c'-1\right)}{2t} \right)^{1/3}.
    \end{equation}
    When $t$ is large relative to $c'$, observe that the value of $\ell$ in \cref{eq:ell_value} satisfies $\ell^2 > c'/t$ and $\ell \leq \frac{1}{2}c' - \frac{3}{2}$. We handle the terms of $f(c', t, \ell)$ individually. Observe that:
    \begin{equation}
        \label{eq:first_term}
        \frac{c't(t-1)^3}{(c't- 1)(c't - 3)} \leq \frac{t(t-1)}{c'}.
    \end{equation}
    Using \cref{eq:ell_value}:
    \begin{align}
        \label{eq:asym_bound}
        c' - 1 - 2 \ell & = c' - 1 - 2\left(\frac{c'(c'-1)}{2t}\right)^{1/3}                         \nonumber \\
                        & = (c'-1) - 2^{2/3} \left( \frac{c'(c'-1)}{t} \right)^{1/3}                 \nonumber \\
                        & =(c'-1)\left( 1 - \frac{2^{2/3} {c'}^{1/3}}{(c'-1)^{2/3} t^{1/3}} \right) \nonumber  \\
                        & = (c'-1) (1 - O_c(t^{-1/3})).
    \end{align}
    By \cref{eq:asym_bound},
    \begin{align}
        \label{eq:second_term}
        (c' - 1 - 2\ell)^2 & = (c' - 1)^2 \left( 1 -  O_c(t^{-1/3}) \right)^2  \nonumber            \\
                           & = (c' - 1)^2 \left( 1 - 2 \cdot O_c(t^{-1/3}) + O_c(t^{-2/3}) \right).
    \end{align}
    For the $(1 - \frac{c'}{\ell^2 t})$ term,
    \begin{equation*}
        \ell^2 t = \left( \frac{c'(c'-1)}{2t} \right)^{2/3} t = \frac{t^{1/3}(c'(c'-1))^{2/3}}{2^{2/3}}.
    \end{equation*}
    Thus,
    \begin{equation}
        \label{eq:third_term}
        1 - \frac{c'}{\ell^2 t} = 1 - c' \cdot \frac{2^{2/3}}{t^{1/3} (c'(c'-1))^{2/3}}
        = 1 - \frac{2^{2/3} {c'}^{1/3}}{(c'-1)^{2/3} t^{1/3}} = 1 -  O_c(t^{-1/3}).
    \end{equation}
    By \cref{eq:first_term,eq:second_term,eq:third_term}, since $t$ is large relative to $c'$,
    \begin{align*}
        f\left(c',t,\ell\right) & \leq \frac{1}{2c'}\left(t(t-1)\right)  \left(c' - 1\right)^{-2} \left( 1 - 2 \cdot O_c(t^{-1/3}) + O_c\left(t^{-2/3}\right) \right)^{-1} \left( 1 - O_c(t^{-1/3})\right)^{-1} \\
                                & =\frac{1}{2c'(c'-1)^2}\left(t(t-1)\right)  \left(1 + O_c(t^{-1/3})\right)^2                                                                                                   \\
                                & =\frac{1}{2c'(c'-1)^2}\left(t(t-1)\right)  \left(1 + O_c(t^{-1/3})\right)                                                                                                     \\
                                & = \left(\frac{1}{c'(c'-1)^2} + O_c(t^{-1/3}) \right) \binom{t}{2}.
    \end{align*}
    Recall that $c' \geq c - \frac{1}{t}$. Therefore,
    \begin{align*}
        f\left(c',t,\ell\right) & \leq \left(\frac{1}{c'(c'-1)^2} + O_c(t^{-1/3}) \right) \binom{t}{2}                                           \\
                                & \leq \left(\frac{1}{(c - \frac{1}{t})(c - \frac{1}{t}-1)^2} + O_c(t^{-1/3}) \right) \binom{t}{2}               \\
                                & = \left(\frac{1}{c(c-1)^2} \left(1 + O_c\left(\frac{1}{t}\right)\right)^3 + O_c(t^{-1/3}) \right) \binom{t}{2} \\
                                & = \left(\frac{1}{c(c-1)^2} \left(1 + O_c\left(\frac{1}{t}\right)\right) + O_c(t^{-1/3}) \right) \binom{t}{2}   \\
                                & =\left(\frac{1}{c(c-1)^2} + O_c(t^{-1/3}) \right) \binom{t}{2}
    \end{align*}
    as desired.
\end{proof}

\textbf{Acknowledgements.}
Thanks to David Wood for helpful comments on an early draft of this paper.

    {
        \fontsize{10pt}{11pt}
        \selectfont
        \bibliographystyle{DavidNatbibStyle}
        \bibliography{ref}
    }

\appendix
\end{document}